\documentclass[11pt]{amsart}
\usepackage[headings]{fullpage}
\usepackage{xcolor,graphicx}
\usepackage[OT2,T1]{fontenc}
\usepackage{hyperref}
\usepackage{caption}
\usepackage{url}
\usepackage{pifont}
\usepackage{amsthm}
\usepackage{amsfonts}
\usepackage{amssymb}
\usepackage[mathscr]{euscript}
\usepackage[all]{xy}
\usepackage{amsmath}
\usepackage{epsfig}
\usepackage{latexsym}
\usepackage{rotating}
\usepackage{stackengine}
\makeatletter
\newcommand{\addresseshere}{%
  \enddoc@text\let\enddoc@text\relax
}
\makeatother

\usepackage[OT2,OT1]{fontenc} \newcommand\cyr
{
\renewcommand\rmdefault{wncyr} \renewcommand\sfdefault{wncyss} \renewcommand\encodingdefault{OT2} \normalfont
\selectfont
}
\DeclareTextFontCommand{\textcyr}{\cyr} 

\newcommand*\wbar[1]{
  \hbox{ \kern-0.2em%
    \vbox{%
      \hrule height 0.5pt  
      \kern0.25ex
      \hbox{%
        \kern-0.15em
        \ensuremath{#1}%
        \kern-0.05em
      }%
    }%
  \kern0.05em}%
} 

\newcommand*\wbarnew[1]{
  \hbox{ \kern-0.2em%
    \vbox{%
      \hrule height 0.5pt  
      \kern0.25ex
      \hbox{%
        \kern-0.35em
        \ensuremath{#1}%
        \kern-0.05em
      }%
    }%
  \kern0.05em}%
}

\newtheorem{theorem}{Theorem}[section]
\newtheorem{lemma}[theorem]{Lemma}

\theoremstyle{definition}

\newtheorem{example}[theorem]{Example}
\theoremstyle{remark}
\newtheorem{remark}[theorem]{Remark}

\title{Hyperbolic Knots and Torsion in Khovanov Homology}
\author[M. Chrisman]{Micah Chrisman}
\author[S. Mukherjee]{Sujoy Mukherjee}
\address{Department of Mathematics, The Ohio State University, Columbus, Ohio, 43210}
\email{chrisman.76@osu.edu}
\email{sujoymukherjee.math@gmail.com}



\begin{document}

\begin{abstract}

In this note, we show that if there is a knot in $S^3$ having $\mathbb{Z}_m$ torsion in its Khovanov homology, then there are infinitely many hyperbolic knots and infinitely many prime satellite knots having $\mathbb{Z}_m$ torsion in their Khovanov homology. As an application, we give the first known examples of hyperbolic knots and links with odd (and other non-$\mathbb{Z}_2$ torsion) in their Khovanov homology.
\end{abstract}

\maketitle

\section{Introduction}

 Let $R$ be a commutative ring with identity and let $L$ be an oriented link in $S^3$. Khovanov's homology theory \cite{Kho} associates a bigraded chain complex over $R$ to $L,$ whose homology is an invariant of the link. The cohomology at bigrading $(i,j)$ is denoted by $\text{Kh}^{i,j}_R(L)$ and the collection of these groups is denoted $\text{Kh}_R(L)$. For $R=\mathbb{Z}$, we write $\text{Kh}(L)$ instead of $\text{Kh}_{\mathbb{Z}}(L)$.  In this case, the torsion part of the abelian group $\text{Kh}^{i,j}(L)$ is the direct sum of finite cyclic groups $\mathbb{Z}_m$. If $\mathbb{Z}_m$ appears as a direct summand of $\text{Kh}^{i,j}(L)$ for some bigrading $(i,j)$, we will say that $\text{Kh}(L)$ has $\mathbb{Z}_m$ torsion.
 
 \
 
 Torsion is a regular occurrence in integral Khovanov homology. Applications of Khovanov homology, however, typically employ its free part only. Famously, the graded Euler characteristic of Khovanov homology gives the Jones polynomial of $L$. It is in this sense that $\text{Kh}(L)$ categorifies the Jones polynomial.  Another application is the Rasmussen invariant \cite{rasmussen}, which gives a lower bound on the smooth slice genus. This invariant is defined only for Khovanov homology with rational coefficients.  The near ubiquity of torsion, together with its absence in application, suggests that the torsion subgroups in Khovanov homology have interesting geometric properties that have yet to be discovered.
 \newline
 
The most commonly observed torsion in Khovanov homology is $\mathbb{Z}_2$ torsion. Every non-split alternating link has either no torsion or only $\mathbb{Z}_2$ torsion (Shumakovitch \cite{Shu,Shu2}). Przytycki and Sazdanovi\'c \cite{PS} have conjectured that the Khovanov homology of a closed 3-braid also has only $\mathbb{Z}_2$ torsion. This was proved for 3-strand torus links by Chandler, Lowrance, Sazdanovi\'c, and Summers, who also established the result for a larger infinite family of closed $3$-braids \cite{Chandler_2021}. Far less common is $\mathbb{Z}_{2^s}$ torsion in Khovanov homology for $s>1$. Even so, $\mathbb{Z}_{2^s}$ torsion of large order is known to exist. The second author, together with Przytycki, Silvero, Wang, and Yang \cite{MPSWY}, found infinitely many links having $\mathbb{Z}_{2^s}$ torsion for $1 \leq s \leq 23$. 
\newline

Also less common than $\mathbb{Z}_2$ torsion is torsion of odd order. Until recently, only a finite number of links were known to have $\mathbb{Z}_m$ torsion for $m$ odd. For example, Bar-Natan \cite{BN07} found that the $T(5,6)$ torus knot has $\mathbb{Z}_3$ and $\mathbb{Z}_5$ torsion in its Khovanov homology. Infinite families of links containing $\mathbb{Z}_3$, $\mathbb{Z}_5$, and $\mathbb{Z}_7$ torsion were constructed in \cite{MPSWY}. In \cite{MS}, Sch\"{u}tz and the second author proved that there are links having $\mathbb{Z}_{p^k}$ torsion for all $k \ge 1$ and $p=3,5,7$. The second author showed in \cite{Muk} that the connected sum $T(5,6)\# T(5,6)$ has $\mathbb{Z}_9$ torsion while $T(5,6)\# T(5,6)\# T(5,6) $ has $\mathbb{Z}_{27}$ torsion. Sch\"{u}tz \cite{schutz} has also shown that the torus knots $T(9,10)$ and $T(9,11)$ have $\mathbb{Z}_9$ and $\mathbb{Z}_{27}$ torsion.
\newline

Now, if $L$ is a non-split prime alternating link that is not a torus link, then $L$ is hyperbolic (Menasco \cite{menasco}). As discussed above, non-split alternating links can only have $\mathbb{Z}_2$ torsion in their Khovanov homology. On the other hand, the examples of non-$\mathbb{Z}_2$ torsion previously mentioned are constructed from torus links, connected sums of torus links, and certain generalized twisted torus links\footnote{This terminology is due to Birman and Koffman \cite{BK}.}. The third type of example is obtained from the braid representation $(\sigma_{m-1} \sigma_{m-2} \cdots \sigma_1)^n$ of $T(m,n)$ by multiplying on the right by $\sigma_1^k$. Here $\sigma_i$ denotes the $i$-th standard generator of the braid group on $m$ strings. Twisted torus links can often be shown to be satellite links (see e.g. \cite{MPSWY}, Figure 3.5). This leads to the natural question: 
\newline 

\centerline{``Are there any hyperbolic links having non-$\mathbb{Z}_2$ torsion in their Khovanov homology?''}
\hspace{1cm}\newline
The purpose of this note is to answer the above question in the affirmative. We prove that there are infinitely many inequivalent hyperbolic knots and links having non-$\mathbb{Z}_2$ torsion in Khovanov homology. Likewise, we show that there are infinitely many inequivalent prime satellite knots having non-$\mathbb{Z}_2$ torsion in their Khovanov homology. Our main tool will be the result of Levine and Zemke \cite{levine_zemke}, that the Khovanov homology functor applied to a ribbon concordance is injective. The results then follow from some well-known facts about knot and link concordance. 
\newline

This paper is organized as follows. In Section \ref{sec_prelim}, we give a brief review of Khovanov homology (Section \ref{subsec_KH}) and ribbon concordance (Section \ref{subsec_ribbon}). Our main results are proved in Section \ref{sec_proofs}.

\section{Preliminaries} \label{sec_prelim}

\subsection{A brief description of Khovanov homology} \label{subsec_KH}

For a comprehensive description of Khovanov homology, the reader is directed to \cite{BN,Tur,Vir}. Our description follows Viro's approach \cite{Vir}. Start with an oriented link diagram $D$. Next, forgetting the orientation, associate to the unoriented link diagram $|D|,$ a bigraded chain complex whose homology is an invariant of the link under regular isotopy. Hence, we first categorify the unreduced Kauffman bracket polynomial \cite{Kau}. Next, the correspondence between this version and the standard version is established by introducing an orientation on the link diagram and calculating its writhe.

\

Choose a diagram $|D|$ for an unoriented link $|L|,$ and let $C(|D|)$ denote its set of crossings. Let $$ks:C(|D|) \longrightarrow \{ A, B \}$$ be a  function defined by associating either an $A$ marker or a $B$ marker to each crossing of $|D|.$ For each function $ks,$ following the conventions in Figure \ref{smoothing_convention}, smooth the crossings of $|D|$ to obtain a system of circles on the plane. This is called a {\it Kauffman state}. Therefore, there is a bijective correspondence between the functions $ks$ and the Kauffman states of $|L|.$ Assuming $|D|$ has $n$ crossings, there are $2^n$ Kauffman states and functions $ks$. Finally, let $\sigma(ks) = | ks^{-1}(A) | - | ks^{-1}(B) |.$

\begin{figure}[ht]
\centering
\includegraphics[scale = 0.15]{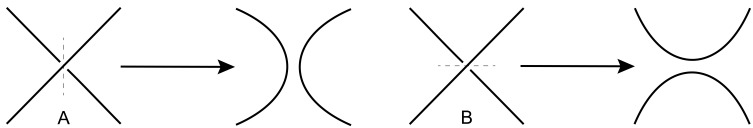}
\caption{Smoothing conventions.}
\label{smoothing_convention}
\end{figure}

Let $c(ks)$ denote the set of circles of the Kauffman state associated to $ks.$ Let $$eks: c(ks) \longrightarrow \{1,-1\}$$ be a function, called an {\it enhanced Kauffman state}, defined by associating either a $1$ or a $-1$ to each circle of the Kauffman state. Furthermore, let $\displaystyle \tau(eks) = \sum_{x \in c(ks)} eks(x).$

\

We are now ready to define the chain modules and the boundary maps. First, let $EKS(|D|)$ denote the set of enhanced Kauffman states of $|D|.$ For $eks \in EKS(|D|),$ let $ks$ be the function corresponding to the Kauffman state obtained from $eks$ by forgetting the enhancements. Define $$ a(eks) = \sigma(ks)\qquad \text{and} \qquad b(eks) = \sigma(ks) + 2\cdot \tau(eks).$$ Next, let $C_{a,b}(|D|)$ be the free $R$-module with generators $$\{ eks \in EKS(|D|) \mid a(eks) = a \text{ and } b(eks) = b \}.$$ Before defining the boundary maps, we need the notion of adjacency of enhanced Kauffman states. Let $eks, eks' \in EKS(|D|).$ Then, $eks'$ is said to be {\it adjacent} to $eks$ if $b(eks) = b(eks'),$ $a(eks') = a(eks) - 2,$ and their corresponding Kauffman states, $ks$ and $ks',$ differ only at one crossing $y \in C(|D|)$ with $ks(y) = A$ and $ks'(y) = B.$

\

After ordering the crossings of $|D|,$ for $b \in \mathbb{Z}$ (that is, for a fixed quantum grading), the boundary maps $\partial_{a,b}:C_{a,b}(|D|) \longrightarrow C_{a-2,b}(|D|)$ are defined as $$ \partial_{a,b}(eks) = \sum_{eks' \in EKS(|D|)} (eks:eks')eks'.$$ Here, if $eks'$ is not adjacent to $eks,$ $(eks:eks') = 0;$ else, $(eks:eks') = (-1)^{\alpha},$ where $\alpha$ is the number of $B$-labeled crossings that occur in $eks$ after the only crossing at which $ks$ and $ks'$ differ. Finally, the homology groups are defined in the usual way, $\displaystyle \text{Kh}_{a,b}^R(|D|) = \frac{ker(\partial_{a,b})}{im(\partial_{a + 2,b})}.$ So far, these groups are invariant under only Reidemeister moves 2 and 3.

\

The path to the standard version of Khovanov cohomology is now quite simple. After reintroducing the given orientation on $|D|,$ and denoting it by $D,$ $$ \text{Kh}^{i,j}_R(D) = \text{Kh}_{w(D) - 2i,3w(D) - 2j}^R(|D|),$$ where $w(D) = \# \text{positive crossings} - \# \text{negative crossings},$ is called the {\it writhe} of $D.$

\subsection{Ribbon Concordance} \label{subsec_ribbon}

Let $K_0,K_1$ be oriented knots in $S^3$. A \emph{concordance from $K_0$ to $K_1$} is a smoothly embedded annulus $A \subset S^3 \times I$ such that $\partial A=K_1 \times \{1\} \sqcup -K_0 \times \{0\}$. A concordance of an $m$-component link $L_0=K_{0,1} \cup \cdots \cup K_{0,m}$ to a $m$-component link $L_1=K_{1,1} \cup \cdots \cup K_{1,m}$ is a collection $A=A_1 \cup \cdots \cup A_m$ of concordances $A_i$ from $K_{0,i}$ to $K_{1,i}$ for $1 \le i \le m$ such that $A_i \cap A_j =\varnothing$ for all $i \ne j$. Since a concordance $A$ between two knots is required to be smooth, it may be assumed that the projection $S^3 \times [0,1] \longrightarrow [0,1]$ restricts to a Morse function on $A$. Each time slice of $A$ corresponding to a regular value of the Morse function is either a knot or a link. The index $0$ critical points, i.e. those of the form $x^2+y^2$, correspond to \emph{births}. This move is of the form $L \longrightarrow L \cup \bigcirc$, where an unknotted component $\bigcirc$ is added to the link diagram $L$ so that it is unlinked from $L$. Index $1$ critical points (i.e. those of the form $x^2-y^2$) correspond to saddle moves. Index $2$ critical points (i.e. those of the form $-x^2-y^2$) are called \emph{deaths}. They are the inverse to births. Concordance is an equivalence relation on oriented knots.
\newline

\begin{figure}[htb]
\begin{tabular}{c} \def\svgwidth{5in}
\begingroup%
  \makeatletter%
  \providecommand\color[2][]{%
    \errmessage{(Inkscape) Color is used for the text in Inkscape, but the package 'color.sty' is not loaded}%
    \renewcommand\color[2][]{}%
  }%
  \providecommand\transparent[1]{%
    \errmessage{(Inkscape) Transparency is used (non-zero) for the text in Inkscape, but the package 'transparent.sty' is not loaded}%
    \renewcommand\transparent[1]{}%
  }%
  \providecommand\rotatebox[2]{#2}%
  \newcommand*\fsize{\dimexpr\f@size pt\relax}%
  \newcommand*\lineheight[1]{\fontsize{\fsize}{#1\fsize}\selectfont}%
  \ifx\svgwidth\undefined%
    \setlength{\unitlength}{561.09006706bp}%
    \ifx\svgscale\undefined%
      \relax%
    \else%
      \setlength{\unitlength}{\unitlength * \real{\svgscale}}%
    \fi%
  \else%
    \setlength{\unitlength}{\svgwidth}%
  \fi%
  \global\let\svgwidth\undefined%
  \global\let\svgscale\undefined%
  \makeatother%
  \begin{picture}(1,0.39681493)%
    \lineheight{1}%
    \setlength\tabcolsep{0pt}%
    \put(0,0.05){\includegraphics[width=\unitlength]{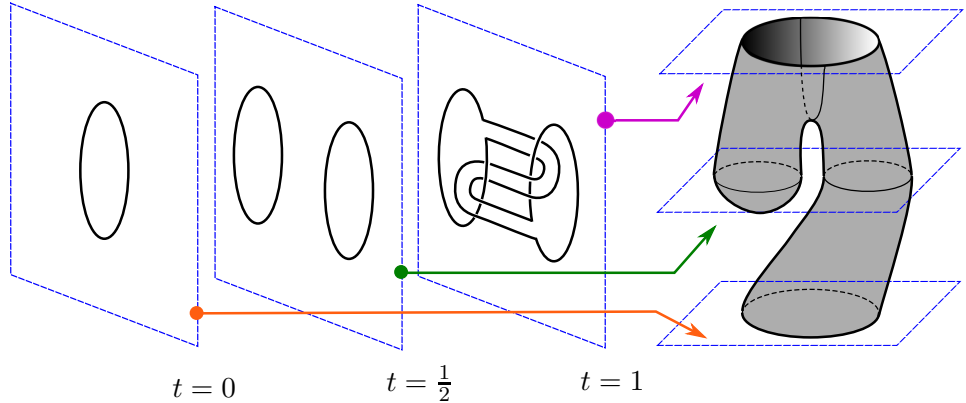}}%
    \put(0.16983545,0.00261941){\color[rgb]{0,0,0}\makebox(0,0)[lt]{\lineheight{40.54999924}\smash{\begin{tabular}[t]{l}$t=0$\end{tabular}}}}%
    \put(0.39386413,0.00954815){\color[rgb]{0,0,0}\makebox(0,0)[lt]{\lineheight{40.54999924}\smash{\begin{tabular}[t]{l}$t=\tfrac{1}{2}$\end{tabular}}}}%
    \put(0.59710667,0.01070292){\color[rgb]{0,0,0}\makebox(0,0)[lt]{\lineheight{40.54999924}\smash{\begin{tabular}[t]{l}$t=1$\end{tabular}}}}%
  \end{picture}%
\endgroup%
 \end{tabular}
\caption{A ribbon concordance consisting of a birth $t=\tfrac{1}{4}$ and a saddle $t=\tfrac{3}{4}$.} \label{fig_ribbon_concordance}
\end{figure} 

A concordance from $K_0$ to $K_1$ is called a \emph{ribbon concordance}\footnote{Our concordances will go from $K_0$ to $K_1$, as in Zemke \cite{zemke}, rather than from $K_1$ to $K_0$, as in Gordon \cite{gordon_ribbon}.} if the Morse function has only critical points of index $0$ and $1$. In this case, we will write $K_0 \le K_1$. Figure \ref{fig_ribbon_concordance} shows a movie of a ribbon concordance with one birth and one saddle. The annulus $A$ is depicted schematically in Figure \ref{fig_ribbon_concordance}, right. Note that for any ribbon knot $K$, there is a ribbon concordance from the unknot to $K$. Ribbon concordance was first studied by Gordon \cite{gordon_ribbon}, who conjectured that ribbon concordance defines a partial ordering on knots in $S^3$. The relation is easily seen to be reflexive and transitive. A proof that $\le$ is antisymmetric was only recently announced by Agol \cite{agol_ribbon}. In the interim, many authors investigated the effect of ribbon concordance on various knot invariants. See, for example, Gordon \cite{gordon_ribbon}, Gilmer \cite{gilmer_ribbon}, Miyazaki \cite{miyazaki_ribbon}, and Silver \cite{silver_ribbon}.
\newline
\newline 
Given an arbitrary concordance $A$ of links from $L_0$ to $L_1$, there is a  linear map $\text{Kh}_R(A):\text{Kh}_R(L_0) \longrightarrow \text{Kh}_R(L_1)$. The map $\text{Kh}_R(A)$ is a functor on the cobordism category of links. For more details, see Bar-Natan \cite{bar_natan_tangle}, Caprau \cite{caprau_tangle}, and Clark-Morrison-Walker \cite{CMW_fixing}. When $A$ is a ribbon concordance from $L_0$ to $L_1$, the following result of Levine and Zemke shows that $\text{Kh}_R(A)$ is an embedding.

\begin{theorem}[Levine-Zemke \cite{levine_zemke}, Theorem 1]\label{thm_levine_zemke} If $A$ is a ribbon concordance from $L_0$ to $L_1$, then $\text{Kh}_R(A):\text{Kh}_R(L_0) \longrightarrow \text{Kh}_R(L_1)$ is an injective and grading-preserving linear map. Furthermore, for any bigrading $(i,j)$, $\text{Kh}^{i,j}_R(L_0)$ is a direct summand of $\text{Kh}^{i,j}_R(L_1)$.
\end{theorem} 

Theorem \ref{thm_levine_zemke} has been shown to hold for many link homology theories. It was first proved for knot Floer homology by Zemke \cite{zemke}. In particular, a ribbon concordance $A$ from a knot $K_0$ to a knot $K_1$ gives an injective linear map $F_A:\widehat{HFK}(K_0) \longrightarrow \widehat{HFK}(K_1)$ that preserves the grading. Here homology is calculated over $\mathbb{Z}_2$. In \cite{daemi_ribbon}, Daemi et al. prove that any flavor of Heegard-Floer homology applied to the 2-fold branch cover of a ribbon concordance also induces an injection. Caprau et al. \cite{caprau_ribbon} showed that Theorem \ref{thm_levine_zemke} holds for the $\mathfrak{sl}(n)$-homology functor for all $n \ge 2$. An axiomatic TQFT approach was used by Kang \cite{kang_ribbon} to prove Theorem \ref{thm_levine_zemke} for a wide class of link homology theories.  

\section{Main Results} \label{sec_proofs} In \cite{myers}, Myers proved that every link is concordant to a hyperbolic link. We will modify this argument to prove our main result. Here we consider Khovanov homology over a fixed ring $R$.

\begin{lemma}[Hyperbolic Subway Coupon Lemma] \label{lemma_hyp_subway} Let $M$ be an $R$-module. If there is an oriented link $L$ with $M$ a submodule of $\text{Kh}_R^{i,j}(L)$ for some grading $(i,j)$, then there are infinitely many distinct hyperbolic links $L'$ such that $M$ is a submodule of $\text{Kh}_R^{i,j}(L')$.
\end{lemma}
\begin{proof} Suppose first that $L=K$ is a knot. Let $J_0$ be a ribbon knot having Alexander polynomial not equal to 1. For example, take $J_0$ to be the stevedore's knot $6_1$. Let $n >0$ be an integer. Since $J_0$ is ribbon, there is a knot $K_n$ such that $K \le K_n$ and $K_n=K \# J_0 \# J_0 \#\cdots \# J_0$ is a connected sum of $K$ with $n$ copies of $J_0$. The Alexander polynomial of $K_n$ is $\Delta_{K_n}(t)=\Delta_{K}(t)\cdot (\Delta_{J_0}(t))^n$ and hence there is an infinite family of knots $K_n$ such that $K \le K_n$ but $K_n$ is not equivalent to $K$.
\newline

Next we modify $K_n$ to a hyperbolic knot $K_n'$ that is ribbon concordant to $K$ and has the same Alexander polynomial as $K_n$. Let $X_n$ be the exterior of $K_n$ in $S^3$. Let $a_1,a_2$ be disjoint arcs of $K_n$. By Myers \cite{myers}, Proposition 6.1, there is a properly embedded arc $\tilde{J}$ in $X_n$ such that $\overline{X_n\smallsetminus N(\tilde{J}})$ is a simple Haken $3$-manifold, where $N(\tilde{J})$ is a regular neighborhood of $\tilde{J}$ in $X_n$. The arc $\tilde{J}$ may be chosen so that one component of $\partial \tilde{J}$ lies in the boundary of a regular neighborhood $N_1$ of $a_1$ while the other component lies in the boundary of a regular neighborhood $N_2$ of $a_2$. Note that $B=N(\tilde{J}) \cup N_1 \cup N_2$ is a 3-ball and that $(B,a_1 \cup a_2)$ is a trivial tangle. See Figure \ref{fig_nbhd}. Let $(T_n,t_n)$ be the $2$-tangle formed by deleting the interior of $B$ from $S^3$ and the interior of the arcs $a_1,a_2$ from $K_n$.
\newline

\begin{figure}[htb]
\begin{tabular}{c} \def\svgwidth{3in}
\begingroup%
  \makeatletter%
  \providecommand\color[2][]{%
    \errmessage{(Inkscape) Color is used for the text in Inkscape, but the package 'color.sty' is not loaded}%
    \renewcommand\color[2][]{}%
  }%
  \providecommand\transparent[1]{%
    \errmessage{(Inkscape) Transparency is used (non-zero) for the text in Inkscape, but the package 'transparent.sty' is not loaded}%
    \renewcommand\transparent[1]{}%
  }%
  \providecommand\rotatebox[2]{#2}%
  \newcommand*\fsize{\dimexpr\f@size pt\relax}%
  \newcommand*\lineheight[1]{\fontsize{\fsize}{#1\fsize}\selectfont}%
  \ifx\svgwidth\undefined%
    \setlength{\unitlength}{279.28918892bp}%
    \ifx\svgscale\undefined%
      \relax%
    \else%
      \setlength{\unitlength}{\unitlength * \real{\svgscale}}%
    \fi%
  \else%
    \setlength{\unitlength}{\svgwidth}%
  \fi%
  \global\let\svgwidth\undefined%
  \global\let\svgscale\undefined%
  \makeatother%
  \begin{picture}(1,0.67178431)%
    \lineheight{1}%
    \setlength\tabcolsep{0pt}%
    \put(0,0){\includegraphics[width=\unitlength]{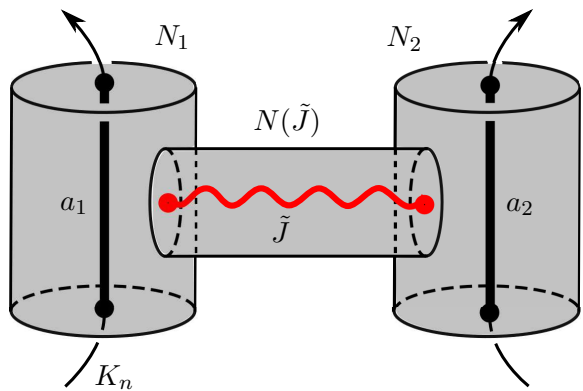}}%
    \put(0.08892417,0.31407538){\color[rgb]{0,0,0}\makebox(0,0)[lt]{\lineheight{40.54999924}\smash{\begin{tabular}[t]{l}$a_1$\end{tabular}}}}%
    \put(0.86736916,0.31146317){\color[rgb]{0,0,0}\makebox(0,0)[lt]{\lineheight{40.54999924}\smash{\begin{tabular}[t]{l}$a_2$\end{tabular}}}}%
    \put(0.45724873,0.25660625){\color[rgb]{0,0,0}\makebox(0,0)[lt]{\lineheight{40.54999924}\smash{\begin{tabular}[t]{l}$\tilde{J}$\end{tabular}}}}%
    \put(0.14900547,0.00844427){\color[rgb]{0,0,0}\makebox(0,0)[lt]{\lineheight{40.54999924}\smash{\begin{tabular}[t]{l}$K_n$\end{tabular}}}}%
    \put(0.25349475,0.60142084){\color[rgb]{0,0,0}\makebox(0,0)[lt]{\lineheight{40.54999924}\smash{\begin{tabular}[t]{l}$N_1$\end{tabular}}}}%
    \put(0.6583906,0.59880861){\color[rgb]{0,0,0}\makebox(0,0)[lt]{\lineheight{40.54999924}\smash{\begin{tabular}[t]{l}$N_2$\end{tabular}}}}%
    \put(0.42590198,0.45252364){\color[rgb]{0,0,0}\makebox(0,0)[lt]{\lineheight{40.54999924}\smash{\begin{tabular}[t]{l}$N(\tilde{J})$\end{tabular}}}}%
  \end{picture}%
\endgroup%
 \end{tabular}
\caption{Construction used in the proof of Lemma \ref{lemma_hyp_subway}.} \label{fig_nbhd}
\end{figure} 

Let $KT=(B',t')$ denote the Kinoshita-Terasaka tangle (see Figure \ref{fig_kt_movie}, bottom right). The $KT$ tangle is prime (Bleiler \cite{bleiler}) and atoroidal (Soma \cite{soma}). Then it follows from Myers \cite{myers}, Lemma 2.5, that the knot $K_n'$ formed from the join $t_n \vee t'$ has simple Haken exterior. Hence, by Thurston \cite{thurston}, it is hyperbolic. The reason for choosing the Kinoshita-Terasaka tangle is that the join $K_n'=t_n \vee t'$ has the same Alexander polynomial as $K_n$ (see Bleiler \cite{bleiler}, Theorem 2.2). This implies that $\{K_0',K_1',\ldots\}$ is an infinite set of distinct hyperbolic knots. 
\newline

The movie shown in Figure \ref{fig_kt_movie} proves that there is a ribbon concordance from $K_n$ to $K_n'=t_n \vee t'$. Hence for each $K_n'$, there is a ribbon concordance $A_n$ from $K$ to $K_n'$. By Theorem \ref{thm_levine_zemke}, $\text{Kh}(A_n)$ is injective. Hence, the linear map $\text{Kh}_R(A_n)$ embeds the given submodule $M$ into $\text{Kh}_R(K_n')$.
\newline

Lastly, suppose that $L$ is a multi-component link and fix one of its components $K$. Perform the construction above on $K$ to obtain knots $K_n$ that are components of a link $L_n$. Myers \cite{myers}, Proposition 6.1, again guarantees that we may choose an arc $\tilde{J}$ beginning and ending on $K_n$ so that the deleting $N(\widetilde{J})$ from the link exterior $X_n$ gives a simple Haken $3$-manifold. As before, replace the trivial tangle $(B,a_1 \cup a_2)$ with the $KT$-tangle to obtain a hyperbolic link $L_n'$. As in the previous case, only one component $K_n'$ has been modified. The set of links $\{L_0',L_1',\ldots\}$ is infinite because the components $K_0',K_1',\ldots$ have distinct Alexander polynomials.
\end{proof}

\begin{figure}[htb]
\begin{tabular}{c} \def\svgwidth{3.8in}
\begingroup%
  \makeatletter%
  \providecommand\color[2][]{%
    \errmessage{(Inkscape) Color is used for the text in Inkscape, but the package 'color.sty' is not loaded}%
    \renewcommand\color[2][]{}%
  }%
  \providecommand\transparent[1]{%
    \errmessage{(Inkscape) Transparency is used (non-zero) for the text in Inkscape, but the package 'transparent.sty' is not loaded}%
    \renewcommand\transparent[1]{}%
  }%
  \providecommand\rotatebox[2]{#2}%
  \newcommand*\fsize{\dimexpr\f@size pt\relax}%
  \newcommand*\lineheight[1]{\fontsize{\fsize}{#1\fsize}\selectfont}%
  \ifx\svgwidth\undefined%
    \setlength{\unitlength}{396.06736183bp}%
    \ifx\svgscale\undefined%
      \relax%
    \else%
      \setlength{\unitlength}{\unitlength * \real{\svgscale}}%
    \fi%
  \else%
    \setlength{\unitlength}{\svgwidth}%
  \fi%
  \global\let\svgwidth\undefined%
  \global\let\svgscale\undefined%
  \makeatother%
  \begin{picture}(1,0.57066434)%
    \lineheight{1}%
    \setlength\tabcolsep{0pt}%
    \put(0,0){\includegraphics[width=\unitlength]{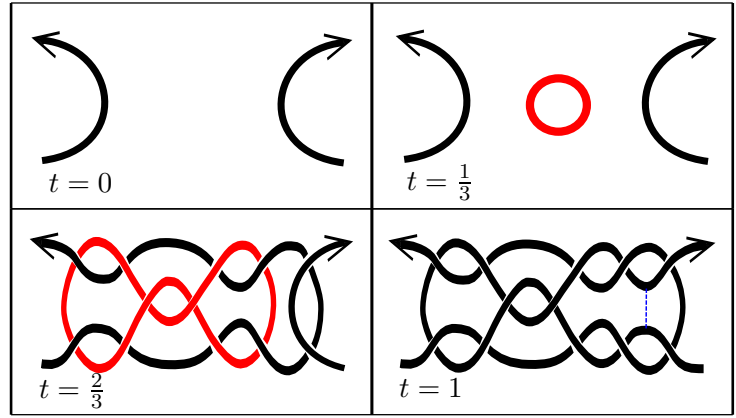}}%
    \put(0.05369327,0.3100066){\color[rgb]{0,0,0}\makebox(0,0)[lt]{\lineheight{40.54999924}\smash{\begin{tabular}[t]{l}$t=0$\end{tabular}}}}%
    \put(0.54679361,0.31301335){\color[rgb]{0,0,0}\makebox(0,0)[lt]{\lineheight{40.54999924}\smash{\begin{tabular}[t]{l}$t=\tfrac{1}{3}$\end{tabular}}}}%
    \put(0.04016309,0.02587258){\color[rgb]{0,0,0}\makebox(0,0)[lt]{\lineheight{40.54999924}\smash{\begin{tabular}[t]{l}$t=\tfrac{2}{3}$\end{tabular}}}}%
    \put(0.53627011,0.0273759){\color[rgb]{0,0,0}\makebox(0,0)[lt]{\lineheight{40.54999924}\smash{\begin{tabular}[t]{l}$t=1$\end{tabular}}}}%
  \end{picture}%
\endgroup%
 \end{tabular}
\caption{A ribbon concordance from the trivial tangle ($t=0$) to the $KT$-tangle ($t=1$). There is one birth (red unknot) and one saddle (blue dashed arc).} \label{fig_kt_movie}
\end{figure} 

In \cite{livingston}, Livingston proved that every knot is concordant to a prime satellite knot. As in Lemma \ref{lemma_hyp_subway}, this argument can be modified to show that corresponding to any knot with torsion in Khovanov homology, there is a prime satellite knot having the same torsion.

\begin{lemma}[Satellite Subway Coupon Lemma] \label{lemma_sat_subway} Let $M$ be an $R$-module. If there is a knot $K$ with $M$ a submodule of $\text{Kh}_R^{i,j}(K)$ for some grading $(i,j)$, then there are infinitely many distinct prime satellite knots $K'$ such that $M$ is a submodule of $\text{Kh}_R^{i,j}(K')$. 
\end{lemma}
\begin{proof} Let $K_n$ be as in the proof of Lemma \ref{lemma_hyp_subway}, so that $K$ is ribbon concordant to $K_n$ and $K,K_n$ are inequivalent. The knot $K_n$ will be the companion of a satellite knot $K_n'$ with the pattern $P$ as shown in Figure \ref{fig_sat conc_movie}, $t=1$. By Livingston \cite{livingston}, Theorem 5.1, this pattern $P$ is a nontrivial knot in the solid torus and has wrapping number greater than $1$. Then any satellite knot $P(K_n)$ is therefore a prime knot in $S^3$ (Livingston \cite{livingston}, Theorem 4.2). The movie shown in Figure \ref{fig_sat conc_movie} shows that there is a ribbon concordance from $K_n$ to $P(K_n)$ and hence from $K$ to $P(K_n)$. Then Theorem \ref{thm_levine_zemke} implies that $M$ is a submodule of $\text{Kh}_R^{i,j}(P(K_n))$. Using the skein relation for the Alexander polynomial, it is easy to prove that $P(K_n)$ and $K_n$ have the same Alexander polynomial (see e.g. Chrisman \cite{chrisman_chp}, Section 5.2). Therefore, the set $\{P(K_0),P(K_1),\ldots\}$ satisfies the conclusion of the theorem.
\end{proof}

\begin{figure}[htb]
\begin{tabular}{c} \def\svgwidth{5.5in}
\begingroup%
  \makeatletter%
  \providecommand\color[2][]{%
    \errmessage{(Inkscape) Color is used for the text in Inkscape, but the package 'color.sty' is not loaded}%
    \renewcommand\color[2][]{}%
  }%
  \providecommand\transparent[1]{%
    \errmessage{(Inkscape) Transparency is used (non-zero) for the text in Inkscape, but the package 'transparent.sty' is not loaded}%
    \renewcommand\transparent[1]{}%
  }%
  \providecommand\rotatebox[2]{#2}%
  \newcommand*\fsize{\dimexpr\f@size pt\relax}%
  \newcommand*\lineheight[1]{\fontsize{\fsize}{#1\fsize}\selectfont}%
  \ifx\svgwidth\undefined%
    \setlength{\unitlength}{475.28453802bp}%
    \ifx\svgscale\undefined%
      \relax%
    \else%
      \setlength{\unitlength}{\unitlength * \real{\svgscale}}%
    \fi%
  \else%
    \setlength{\unitlength}{\svgwidth}%
  \fi%
  \global\let\svgwidth\undefined%
  \global\let\svgscale\undefined%
  \makeatother%
  \begin{picture}(1,0.38034968)%
    \lineheight{1}%
    \setlength\tabcolsep{0pt}%
    \put(0,0){\includegraphics[width=\unitlength]{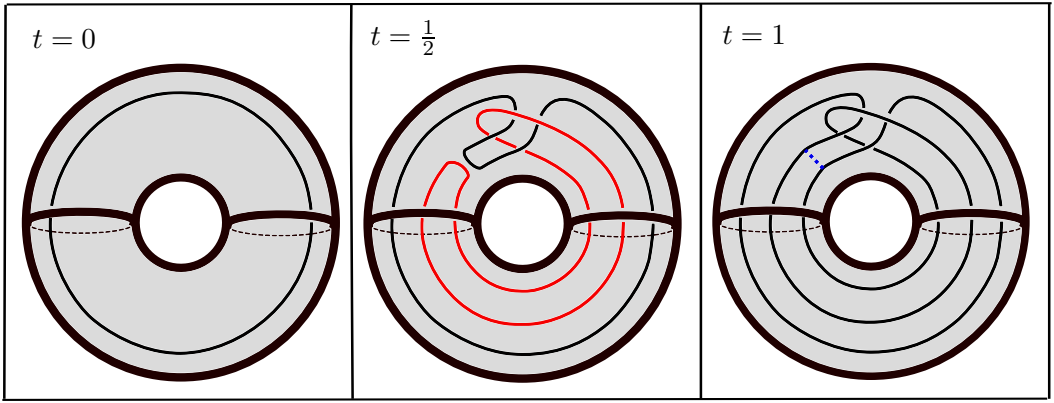}}%
    \put(0.02869143,0.33691693){\color[rgb]{0,0,0}\makebox(0,0)[lt]{\lineheight{40.54999924}\smash{\begin{tabular}[t]{l}$t=0$\end{tabular}}}}%
    \put(0.35026913,0.33930785){\color[rgb]{0,0,0}\makebox(0,0)[lt]{\lineheight{40.54999924}\smash{\begin{tabular}[t]{l}$t=\tfrac{1}{2}$\end{tabular}}}}%
    \put(0.68499678,0.3405033){\color[rgb]{0,0,0}\makebox(0,0)[lt]{\lineheight{40.54999924}\smash{\begin{tabular}[t]{l}$t=1$\end{tabular}}}}%
  \end{picture}%
\endgroup%
 \end{tabular}
\caption{A ribbon concordance from the core circle ($t=0$) to Livingston's pattern ($t=1$). There is one birth (red unknot) and one saddle (blue dashed arc).} \label{fig_sat conc_movie}
\end{figure}

\begin{remark} The constructions used in Lemmas \ref{lemma_hyp_subway} and \ref{lemma_sat_subway} apply to any link homology theory for which Theorem \ref{thm_levine_zemke} is true. Hence, they hold for knot Floer homology \cite{zemke} and $\mathfrak{sl}(n)$-homology \cite{caprau_ribbon}.
\end{remark} 
 
\begin{theorem} There are infinitely many hyperbolic links and infinitely many prime satellite knots having non-$\mathbb{Z}_2$ torsion in their Khovanov homology. In particular, the following are true.
\begin{enumerate}
    \item For $m=3,4,5,7,8,9,27$, there are infinitely many hyperbolic knots and infinitely many prime satellite knots having $\mathbb{Z}_m$ torsion in their Khovanov homology. 
    \item For $1\le k$, $1 \leq s \leq 23$, and $m=2^s,3^k,5^k,7^k$, there are infinitely many multi-component hyperbolic links having $\mathbb{Z}_m$ torsion in their Khovanov homology.
\end{enumerate}
\end{theorem} 
\begin{proof} By Lemmas \ref{lemma_hyp_subway} and \ref{lemma_sat_subway}, it suffices to show that there exists some link with the desired torsion. For (1), knots with these torsion groups can be found in \cite{BN07}, \cite{Muk}, \cite{MPSWY}, and \cite{schutz}. For (2), multi-component links with these torsion groups can be found in \cite{MPSWY} or \cite{MS}.
\end{proof}

\begin{figure}[htb]
\begin{tabular}{cc} \def\svgwidth{2in}
\begingroup%
  \makeatletter%
  \providecommand\color[2][]{%
    \errmessage{(Inkscape) Color is used for the text in Inkscape, but the package 'color.sty' is not loaded}%
    \renewcommand\color[2][]{}%
  }%
  \providecommand\transparent[1]{%
    \errmessage{(Inkscape) Transparency is used (non-zero) for the text in Inkscape, but the package 'transparent.sty' is not loaded}%
    \renewcommand\transparent[1]{}%
  }%
  \providecommand\rotatebox[2]{#2}%
  \newcommand*\fsize{\dimexpr\f@size pt\relax}%
  \newcommand*\lineheight[1]{\fontsize{\fsize}{#1\fsize}\selectfont}%
  \ifx\svgwidth\undefined%
    \setlength{\unitlength}{181.9373942bp}%
    \ifx\svgscale\undefined%
      \relax%
    \else%
      \setlength{\unitlength}{\unitlength * \real{\svgscale}}%
    \fi%
  \else%
    \setlength{\unitlength}{\svgwidth}%
  \fi%
  \global\let\svgwidth\undefined%
  \global\let\svgscale\undefined%
  \makeatother%
  \begin{picture}(1,0.59520011)%
    \lineheight{1}%
    \setlength\tabcolsep{0pt}%
    \put(0,0){\includegraphics[width=\unitlength]{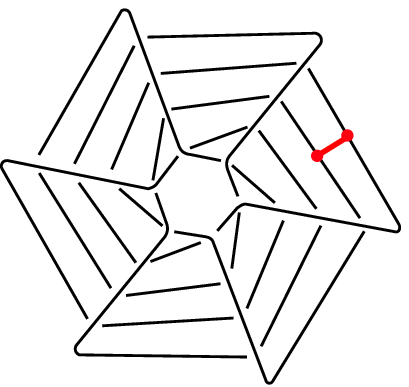}}%
    \put(0.9,0.6){\color[rgb]{0,0,0}\makebox(0,0)[lt]{\lineheight{40.54999924}\smash{\begin{tabular}[t]{l}$\widetilde{J}$\end{tabular}}}}%
  \end{picture}%
\endgroup%
 & \def\svgwidth{3in}
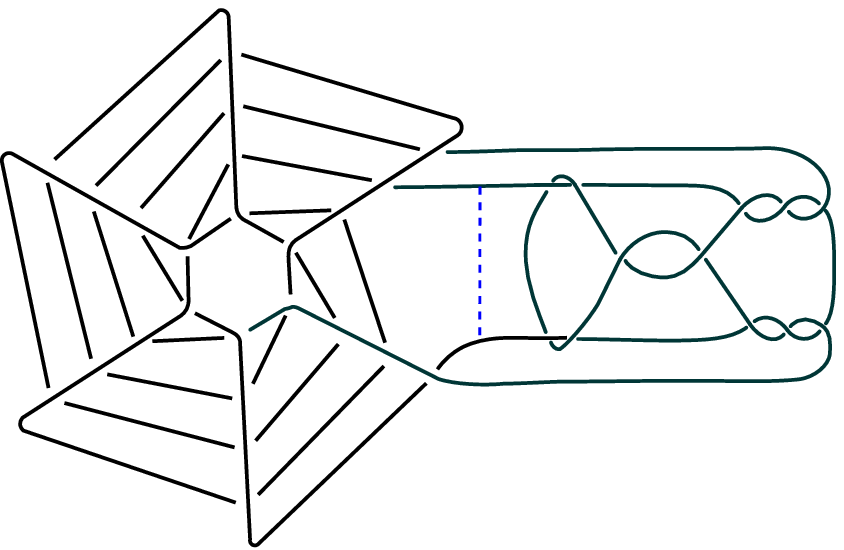 \end{tabular}
\caption{$T(5,6)$ (left) is ribbon concordant to the hyperbolic knot $K$ (right).} \label{fig_t56_plus}
\end{figure}

\begin{example} An example of a hyperbolic knot $K$ with odd torsion in its Khovanov homology is given in Figure \ref{fig_t56_plus}, right. According to SnapPy \cite{SnapPy}, the knot is hyperbolic with volume $\approx 20.9424081371$. The knot $K$ can be obtained from the torus knot $T(5,6)$ (Figure \ref{fig_t56_plus}, left) using the construction of Lemma \ref{lemma_hyp_subway}. The red arc in Figure \ref{fig_t56_plus} is the arc $\widetilde{J}$, as in Figure \ref{fig_nbhd}. The saddle occurring in the ribbon concordance (see Figure \ref{fig_kt_movie}) is shown as a blue dashed arc in Figure \ref{fig_t56_plus}, right. The Khovanov homology of $T(5,6)$ is given in Figure \ref{fig_t56} while that of $K$ is given in Figure \ref{fig_big_kh}. In both tables, we write $A_b$ in position $(i,j)$ of the table if $\mathbb{Z}^A_b=\mathbb{Z}_b \oplus \cdots \oplus \mathbb{Z}_b$ ($A$-times) is a direct summand of $\text{Kh}^{i,j}(L)$. For $\mathbb{Z}^A$, we simply write $A$. The shaded entries in Table \ref{fig_big_kh} correspond to the nontrivial bigradings in the Khovanov homology of $T(5,6)$, so that the conclusion of Theorem \ref{thm_levine_zemke} is evident. Observe that $K$ has both $\mathbb{Z}_3$ and $\mathbb{Z}_5$ torsion and hence the odd torsion is preserved.
\end{example}

\subsection*{Acknowledgments} Lemmas \ref{lemma_hyp_subway} and \ref{lemma_sat_subway} were appropriately discovered while the authors were walking to redeem a buy-one-get-one-free coupon from their nearest (now closed) \emph{Subway} restaurant (1952 N High St, Columbus, OH). The authors would like to thank \emph{Subway} for their generosity.

\

The computational data in this note was obtained using JavaKh-v2 written by Scott Morrison. It is an update of Jeremy Green’s JavaKh-v1 written under the supervision of Dror Bar-Natan.

\

The authors would like to thank Rhea Palak Bakshi, Carmen Caprau, and J\'{o}zef H. Przytycki for their valuable insights and suggestions. The second author is grateful to Hanna Makaruk and Robert Owczarek for organizing the AMS Special Sessions on Inverse Problems for many years and giving him the opportunity to participate.

\

The first author recognizes the support of The Ohio State University, Marion Campus, who provided funds and release time. The second author was supported by the American Mathematical Society and the Simons Foundation through the AMS-Simons Travel Grant.

\bibliographystyle{plain}
\bibliography{odd_KH}
\addresseshere

\hspace{1cm}
\newline

\begin{figure}[h]
\begin{minipage}[ht]{0.64\textwidth}
\begin{center}
\includegraphics[scale = 0.59]{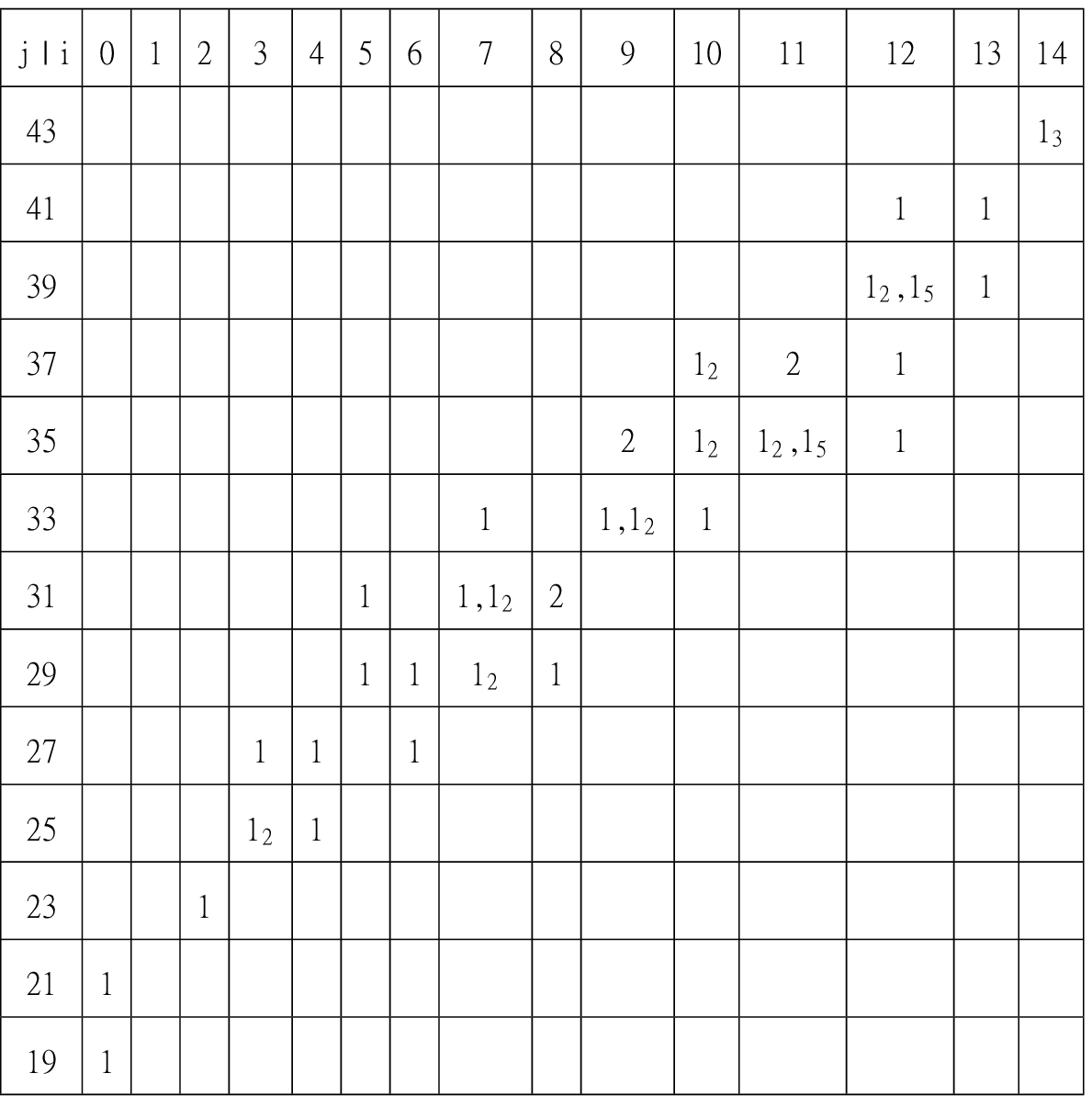}
\end{center}
\end{minipage} 
\begin{minipage}[ht]{0.35\textwidth}
\begin{center}
\def\svgwidth{2.1in}
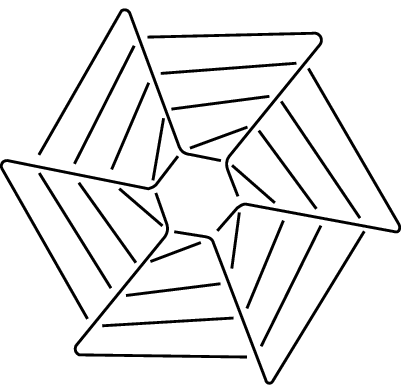
\end{center}
\end{minipage}
\caption{The Khovanov homology (left) of the $T(5,6)$ torus knot (right).} \label{fig_t56}
\end{figure}
\newpage
\begin{figure}[hb]
\rotatebox{90}{
\scalebox{.58}{\psfig{figure=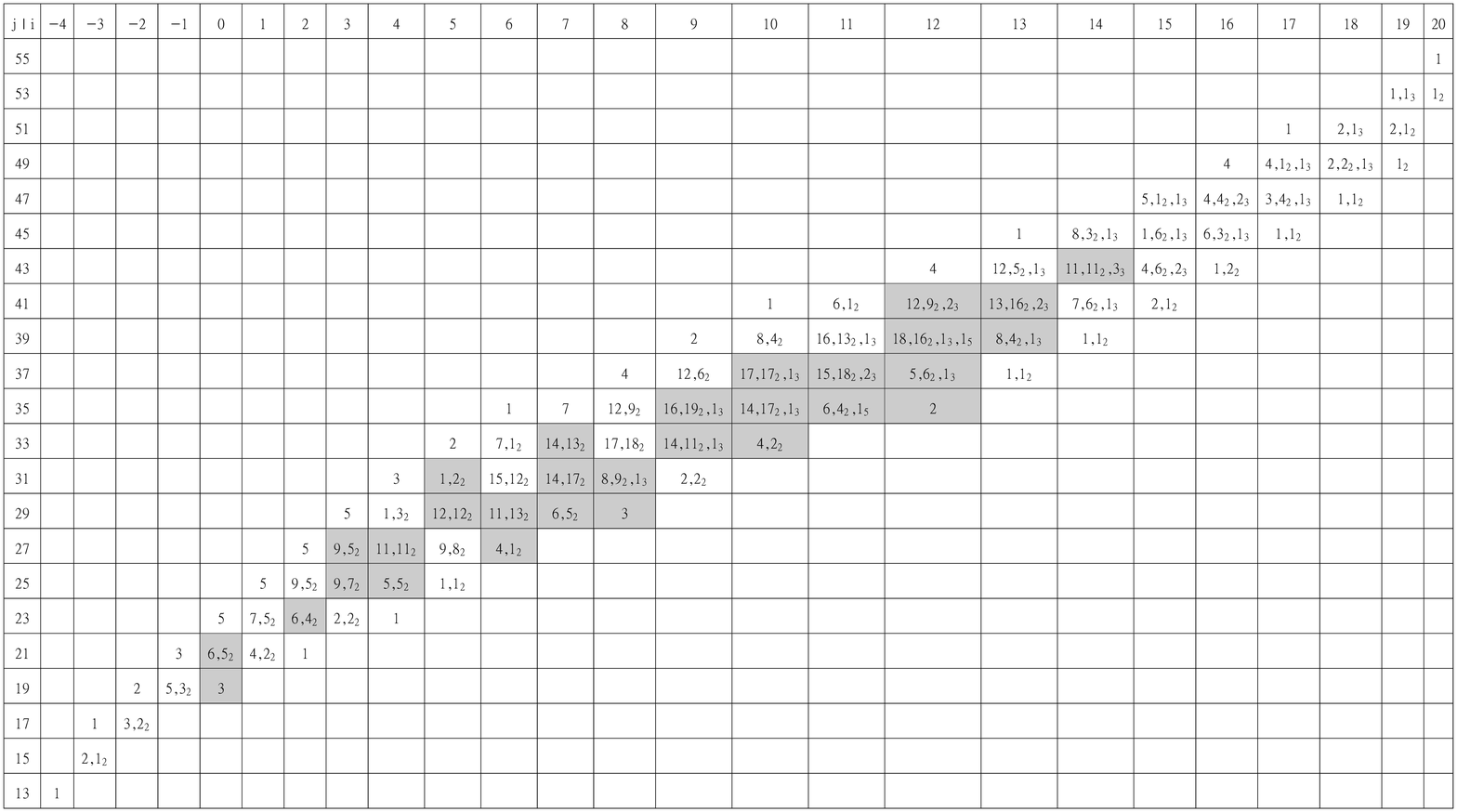}}}
\caption{The Khovanov homology of the knot from Figure \ref{fig_t56_plus}.} \label{fig_big_kh}
\end{figure}
\end{document}